\documentclass[a4paper,12pt]{amsart}
\usepackage[foot]{amsaddr}
\usepackage[T1]{fontenc}
\usepackage[utf8]{inputenc}
\usepackage{palatino}
\usepackage{amsmath, amssymb,mathtools}
\usepackage{soul}
\usepackage[usenames,dvipsnames]{xcolor}
\usepackage{graphicx}
\usepackage{subfigure}
\usepackage{minitoc}
\usepackage{tikz}
\usepackage{enumerate}
\usepackage[margin=0.96 in]{geometry}
\usepackage{bbm}
\usepackage[numbers,sort&compress]{natbib}
\usepackage[colorlinks=true]{hyperref}
\hypersetup{urlcolor=blue, citecolor=red}
\usepackage{microtype}
\usepackage{crossreftools}

\DeclarePairedDelimiter{\abs}{\lvert}{\rvert}
\DeclarePairedDelimiter{\norm}{\lVert}{\rVert}

\DeclareMathAlphabet{\mathup}{OT1}{\familydefault}{m}{n}
\newcommand{\dx}[1]{\mathop{}\!\mathup{d} #1}

\DeclarePairedDelimiter{\prt}{(}{)}
\DeclarePairedDelimiter{\brk}{[}{]}

\newcommand{\R}{{\mathbb R}}
\newcommand{\Rd}{{\mathbb R^d}}
\newcommand{\curlyH}{\mathcal{H}}

\usepackage{amsthm}
\theoremstyle{plain}
\newtheorem{theorem}{Theorem}[section]
\newtheorem{lemma}[theorem]{Lemma}
\newtheorem{proposition}[theorem]{Proposition}
\newtheorem{corollary}[theorem]{Corollary}

\theoremstyle{remark}

\newtheorem{definition}[theorem]{\bf Definition}

\renewcommand{\i}{^{(i)}}
\newcommand{\1}{^{(1)}}
\newcommand{\2}{^{(2)}}

\newcommand{\ds}{\displaystyle}
\newcommand{\ddt}{\frac{\dx{}}{\dx t}}
\newcommand{\partialt}[1]{\frac{\partial #1}{\partial t}}

\newcommand{\fpartial}[1]{\frac{\partial}{\partial #1}}
\newcommand{\grad}{\nabla}
\renewcommand{\div}{\nabla\cdot}

\title[Nonlocal approximation of anisotropic diffusion]{Nonlocal approximation of an anisotropic cross-diffusion system}

\author{Tomasz D\k{e}biec$^{1}$}
\author{Markus Schmidtchen$^{2}$}

\address{$^{1}$ Institute of Applied Mathematics and Mechanics, University of Warsaw, Banacha 2, 02-097 Warsaw, Poland (t.debiec@mimuw.edu.pl).}
\address{$^{2}$ Institute of Scientific Computing, Faculty of Mathematics, TU Dresden
	(mmarkus.schmidtchen@tu-dresden.de).}

\begin{document}

\maketitle
\begin{abstract}
    Localisation limits and nonlocal approximations of degenerate parabolic systems have experienced a renaissance in recent years. However, only few results cover anisotropic systems. This work addresses this gap by establishing the nonlocal-to-limit for a specific anisotropic cross-diffusion system encountered in population dynamics featuring phase-separation phenomena, i.e., internal layers between different species. A critical element of the proof is an entropy dissipation identity, which we show to hold for any weak solution.
\end{abstract}{}

\vskip .4cm
\begin{flushleft}
    \noindent{\makebox[1in]\hrulefill}
\end{flushleft}
	Anisotropic Parabolic-Hyperbolic Cross-Diffusion Systems;\\[-2.em]
\begin{flushright}
    \noindent{\makebox[1in]\hrulefill}
\end{flushright}
\vskip 1.5cm

\section{Introduction}
Nonlocal repulsive interaction equations and nonlinear diffusion equations are popular choices to model repulsive behaviour in biomathematics and population dynamics \cite{CHS2017, MT2015, VS2015, HB2020, PHP2024}, in pedestrian flow or traffic flow \cite{CCS2019, GB2018, Gib2020}. Due to their applicability to real-world scenarios, understanding the fine relationship between nonlocal dispersal and nonlinear diffusion has gained considerable traction in recent years.

Nonlocal models often describe interactions through convolution terms, archetypically of the form
\begin{align*}
    \partialt n = \nabla (n \nabla W \star n),
\end{align*}
while local models, such as nonlinear diffusion, are typically governed by parabolic equations with the quadratic porous medium equation
\begin{align*}
    \partialt n = \nabla (n \nabla n),
\end{align*}
being the local counterpart of the nonlocal equation above. The connection between these two models was first observed in a deterministic setting in \cite{DM1990, LM2001}. Related scaling limits were studied for stochastic particle systems in \cite{Oel1990, Oel2001, Phi2007, FP2008}.

More recently, the subject of localisation limits and nonlocal particle approximations of local equations has received considerable attention -- in \cite{CCP2019}, the authors propose a blob method and establish the limit as a Gamma-convergence result in the quadratic case. In \cite{BE2022}, a similar result is proven for the quadratic porous medium equation and a specific class of quadratic cross-diffusion systems. Using classical parabolic theory, rather than relying on the Wasserstein gradient-flow structure, \cite{DHPP2023} obtained a similar result for specific cross-diffusion systems. Most recently, \cite{CEW2023} managed to establish the localisation limit for a broad class of degenerate parabolic equations for nonlinearities with power-law growth, and \cite{CESW2024, CJT2024} treat the case of linear and fast diffusion equations. In \cite{DIS2024}, the authors show the convergence of a particle method without the double-convolution structure typically encountered in localisation limits. Here, the particles interact through rescaled repulsive Morse interactions and the interaction kernel converges to a Dirac.

In contrast to the aforementioned second-order models, recent works have also established localisation limits for Cahn-Hilliard-type systems \cite{CES2023, ES2023_2}.

A different type of localisation limit has been considered for evolution equations on graphs introduced in \cite{EPSS2021}. However, instead of localising the interaction kernel as in the literature mentioned above, \cite{EHS2023} considers localisation of the connectivity of the graph, obtaining a local continuity equation with nonlocal interactions -- reminiscent of the usual aggregation equation, but with an anisotropy tensor representing the graph structure. The limit equation is, in a sense, the starting point of this paper, where we study the localisation limit of anisotropic interaction equations for systems of interacting species.\\

\subsection{Goal of this paper}

The type of cross-diffusion system we have in mind was introduced initially in \cite{BT1983} as a model for communicable diseases and in \cite{GP1984} to model population dynamics accounting for anti-overcrowding effects. Here, the highest-order terms of the system were of the form
\begin{align}
    \label{eq:xdiff-system-isotropic}
    \partialt {n\i} = D\i \nabla \cdot (n\i \nabla (n\1 + n\2)), \qquad i=1,2.
\end{align}
An interesting feature of this model is the emergence of internal layers in systems whose subpopulations $n\1$ and $n\2$ are segregated initially. This conservation of phase separation property was further studied in a sequence of works \cite{bertsch1987degenerate, bertsch1987interacting, bertsch1985interacting}, and proposed as a tissue growth model with contact inhibition in \cite{bertsch2012nonlinear, mimura2010free}. In recent years, more general well-posedness results were given \cite{CFSS2017, GPS2019, PX2020, LX2021, Jac2021, Jac2022} with strategies covering optimal transport, parabolic regularity theory, and energy-dissipation methods. The existence of Turing-type patterns was studied in \cite{galiano2020turing}. 
Additionally, \cite{DDMS2024, DMS_phenotype} were able to provide existence results for this class of degenerate parabolic system based on an approximation procedure involving a system of nonlocally interacting species. In \cite{jungel2024nonlocal}, the authors apply the strategy introduced in the aforementioned works to a class of anisotropic systems on bounded domains with an additional $\Delta n\i$-term. However, due to strong assumptions on the diffusion matrix, the special case of an anisotropic Darcy-type relation $v= A \nabla p$ is not covered. In this paper, we fill this gap.\newline

\textbf{The system.} Throughout, we consider the following system of equations posed in $\Rd\times(0, T)$
\begin{subequations}
    \label{eq:System}
\begin{align}
    \label{eq:System-species}
	\ds {\partialt {n\i_\nu}} &= \div(n\i_\nu A\nabla m_\nu) + n\i_\nu G\i(n_\nu),\quad i=1,2,
\end{align}
coupled through 
\begin{align}
    \label{eq:System-Brinkman}
    -\nu \div(A\nabla m_\nu) + m_\nu &= n_\nu,
\end{align}
\end{subequations}
where $n\i_\nu = n\i_\nu(x,t)$ denotes the density of species $i=1,2$, and we introduced the notation $n_\nu = n\1_\nu + n\2_\nu$ to denote the joint population density. Note that we recover the viscoelastic tissue growth model for two species introduced in \cite{DS2020, DPSV2021} when choosing $A$ as the identity, as well as Eq. \eqref{eq:xdiff-system-isotropic} in the vanishing viscosity limit $\nu = 0$. Moreover, upon adding the equations for the two densities, we readily obtain the following equation for the total population density
\begin{equation}
    \label{eq:SumEquation}
    \partialt {n_\nu} = \div(n_\nu A \nabla m_\nu) + n\1_\nu G\1(n\nu) + n\2_\nu G\2(n_\nu).
\end{equation}
which played an important role in the localisation limits in the isotropic regime \cite{DDMS2024, ES2023, DMS_phenotype}.\newline

\textbf{The anisotropy tensor.} We make the following assumptions about $A=A(x,t)\in\R^{d\times d}$: 
\begin{itemize}
\setlength\itemsep{.5em}
    \item $A \in W^{1,\infty}(0,T;L^\infty(\Rd)) \cap L^{\frac{2q}{q-d}}(0,T;\dot{W}^{1,q}(\Rd))$, for some $q\geq d$,
    \item $A(x,t)$ is symmetric for a.e.\ $(x,t)\in\Rd\times(0,T)$,
    \item $A(x,t)$ is $\lambda$-uniformly elliptic: $\exists \lambda >0: \xi^T A(x,t) \xi \geq \lambda |\xi|^2,$ for all $\xi\in\R^d$ and a.e. $(x,t)$.\\
\end{itemize}

\textbf{The growth rate.} The functions $G\i:\R\to\R$, $i=1,2$, are assumed to satisfy:
\begin{itemize}
\setlength\itemsep{.5em}
    \item Regularity: $G\i \in C^1(\R)$,
    \item Monotonicity: $\prt{G\i}' \leq -\alpha \leq 0$ for some $\alpha>0$,
    \item Homeostatic pressure: $\exists \bar n\i>0: G\i(\bar n\i)=0$. We set $\bar n:=\max{\{\bar n\1, \bar n\2\}}$.\\
\end{itemize}

\textbf{The initial data.} Let $n^{(i),\mathrm{in}}$, $i=1,2$, be nonnegative functions such that
\begin{itemize}
\setlength\itemsep{.5em}
    \item $n^{(i),\mathrm{in}} \in L^1\cap L^\infty(\Rd)$,
    \item $0\leq n^{(i),\mathrm{in}} \leq \bar n\i$,
    \item $|x|^2\prt{n^{(1),\mathrm{in}} + n^{(2),\mathrm{in}}} \in L^1(\Rd)$.\\
\end{itemize}

\textbf{Main result.} Under the above assumptions, we shall prove the vanishing viscosity limit, $\nu\to 0$.  Formally, System~\eqref{eq:System} becomes 
\begin{align}
\label{eq:InviscidSystem}
\left\{
\begin{array}{rl}
	\ds \partialt {n_0\i} &= \div(n_0\i A\nabla n_0) + n_0\i G\i(n_0),\quad i=1,2,\\
    n_0&=n_0\1+n_0\2,
\end{array}
\right.
\end{align}
a degenerate parabolic cross-diffusion system constituting the anisotropy analogue of Eq. \eqref{eq:xdiff-system-isotropic}.
Specifically, our main result establishes this limit rigorously at the level of weak solutions.
\begin{theorem}
    \label{thm:Main}
    For $\nu>0$, let $(n_\nu\1, n_\nu\2, m_\nu)$ be a weak solution of System~\eqref{eq:System} with initial data $(n^{(1),\mathrm{in}}, n^{(2),\mathrm{in}})$. There exists a subsequence such that
    \begin{alignat*}{2}
        n_\nu\i &\stackrel{\star}{\rightharpoonup} n_0\i \qquad &&\text{weakly$^\star$ in } L^\infty(0,T;L^1\cap L^\infty(\Rd)),\;i=1,2, \\
        n_\nu &\to n_0 \qquad &&\text{strongly in } L^2(0,T;L^2(\Rd)),\\
        m_\nu &\rightharpoonup n_0 \qquad &&\text{weakly in } L^2(0,T;H^1(\Rd)),\\
        m_\nu &\to n_0 \qquad &&\text{strongly in } L^2(0,T;L^2(\Rd)),
    \end{alignat*}
    where $n_0 = n_0\1+n_0\2$ and $(n\1_0, n\2_0)$ is a weak solution to System~\eqref{eq:InviscidSystem} with initial data $(n^{(1),\mathrm{in}}, n^{(2),\mathrm{in}})$.
\end{theorem}

\textbf{Outline of the paper.} The proof of Theorem~\ref{thm:Main} is split into several steps. In Section~\ref{sec:Existence}, we discuss the existence of weak solutions to System~\eqref{eq:System} and some regularity properties of the velocity field. Then, in Section~\ref{sec:Uniform}, we establish certain necessary $\nu$-uniform estimates, allowing us to extract weakly convergent subsequences. 
In Section~\ref{sec:Entropy}, we establish the entropy identity satisfied by any weak solution of Eq. \eqref{eq:System}. 
This is achieved by an appropriate mollification procedure. Subsequently, in Section~\ref{sec:StrongConvergence}, we leverage the uniform control over the entropy-dissipation term to deduce strong convergence of the total density and weak convergence of $(\nabla m_\nu)_\nu$. 
These properties are sufficient to pass to the limit in the equation for the total population density. Furthermore, they allow us to derive a corresponding entropy identity for the limit equation, see Section~\ref{sec:Limit}. Finally, in Section~\ref{sec:StrongVelocity}, we compare the two entropy identities to upgrade the weak convergence of $(\nabla m_\nu)_\nu$ and deduce strong convergence of the velocity. Due to the lack of strong convergence of the individual densities, this is necessary to pass to the limit in the weak formulation in Definition~\ref{def:WeakBrinkman}.

Formally, the main idea of proof can be summarised as follows. Testing Eq. \eqref{eq:SumEquation} with $\log n$,  we obtain an entropy identity, see Eq.~\eqref{eq:entropy_equality}, with dissipation term
\begin{equation*}
    -\int_0^T\!\!\!\int_\Rd n_\nu\div(A\nabla m_\nu) \dx x \dx t \leq C,
\end{equation*}
which, due to the anisotropic Brinkman equation and the ellipticity of $A$, controls the $L^2$-norm of $\nabla m_\nu$. This implies that $\nabla m_\nu\rightharpoonup\nabla n_0$ and $m_\nu\to n_0$. Moreover,
\begin{equation*}
    \nu \int_0^T\!\!\!\int_\Rd n_\nu \div(A\nabla m_\nu \dx x \dx t \to 0,
\end{equation*}
which, again due to Brinkman's equation, implies that $n_\nu$ and $m_\nu$ have the same limit in $L^2$, i.e., $n_\nu\to n_0$. Finally, by comparing the entropy identities at $\nu>0$ and at $\nu=0$, see Eq. ~\eqref{eq:DarcyEntropy}, we deduce
\begin{equation*}
    \int_0^T\!\!\!\int_\Rd \nabla m_\nu \cdot A\nabla m_\nu \dx x \dx t\leq \int_0^T\!\!\!\int_\Rd \nabla n_0 \cdot A\nabla n_0 \dx x \dx t + o(\nu).
\end{equation*}
The resulting upper semi-continuity property, combined with the lower semi-continuity of the $L^2$-norm, yields the convergence $\norm{\nabla m_\nu}_{L^2}\to\norm{\nabla n_0}_{L^2}$. Together with weak convergence, this implies strong $L^2$-convergence of $\nabla m_\nu$.

\section{Existence of solutions}
\label{sec:Existence}
In this section, we state the definition of weak solutions to System~\eqref{eq:System}, assert their existence, and discuss regularity properties of the velocity field $-A\nabla m_\nu$. 
\begin{definition}
    \label{def:WeakBrinkman}
    We say that the triple $(n\1_\nu, n\2_\nu, m_\nu)$ is a \emph{weak solution} to System~\eqref{eq:System} with initial data $(n_\nu^{(1), \mathrm{in}}, n_\nu^{(2), \mathrm{in}})$ if $n\i_\nu \in L^\infty(0,T; L^1 \cap L^\infty(\Rd))$ are nonnegative and there holds
	\begin{align*}
		-\int_0^T \int_\Rd & n\i_\nu \partialt \varphi \dx x \dx t
		+ \int_0^T \int_\Rd n\i_\nu \prt{A\nabla m_\nu} \cdot \nabla \varphi \dx x \dx t
		\\
		& = \int_0^T \int_\Rd \varphi n\i_\nu G\i(n_\nu) \dx x \dx t +\int_\Rd \varphi(x,0)n^{(i),\mathrm{in}}(x) \dx x ,
	\end{align*}
	for $i=1,2$, for any $\varphi \in C_{c}^{\infty}(\Rd \times [0,T))$, as well as
	\begin{align*}
		- \nu\div(A\nabla m_\nu) + m_\nu = n_\nu,
	\end{align*}
	almost everywhere in $\Rd\times(0,T)$.
\end{definition}

As the main focus of this work is on the localisation limit, let us simply claim that for any $\nu>0$ there exists at least one weak solution $(n\1, n\2, m)$ of System~\eqref{eq:System}, according to Definition~\ref{def:WeakBrinkman}, with the following properties
\begin{align*}
    0 \leq n\i \leq \bar n,\quad 0\leq m \leq \bar n,
\end{align*}
and, for $n := n\1 + n\2$,
\begin{align*}
    n \in C([0,T]; L^2(\Rd)),\quad \partialt n \in L^2(0,T;H^{-1}(\Rd)).
\end{align*}
Clearly, the upper bound $\bar n$ is independent of $\nu$. We will revisit this uniform bound in Lemma~\ref{lem:LpBounds}, providing a formal proof. While we do not give a proof for existence, let us note that it can be obtained via a standard parabolic approximation. We refer the reader to~\cite{DDMS2024, DMS_phenotype, ES2023, BelgacemJabin} for details of an appropriate approximation scheme. As shown in the following subsection, despite the presence of the anisotropy tensor, the velocity field $-A\nabla m$ possesses enough regularity to make the cited approaches work in our current case. Specifically, 
\begin{subequations}
\label{eq:reg-velo}
\begin{align}
    \label{eq:reg-velo-1}
    A\nabla m \in L^2(0,T;H^1(\Rd)),
\end{align}
as well as
\begin{align}
    \label{eq:reg-velo-2}
    \nabla \cdot( A \nabla m) \in L^\infty(0,T;L^\infty(\Rd)).
\end{align}
\end{subequations}

\subsection{Regularity of the velocity}
Let us establish the aforementioned regularity properties, Eq. Eq. \eqref{eq:reg-velo}, of the velocity field,  $-A\nabla m_\nu$. This requires establishing sufficient regularity of the velocity potential, $m_\nu$, first. At this point, we make no claims as to the uniformity of these bounds with respect to $\nu$.

First, multiplying the Brinkman equation, Eq. \eqref{eq:System-Brinkman} by $m_\nu$ and integrating by parts, we have
\begin{align*}
    \nu\int_\Rd \nabla m_\nu \cdot A\nabla m_\nu \dx x + \int_\Rd m_\nu^2 \dx x = \int_\Rd m_\nu n_\nu \dx x \leq \frac12\int_\Rd m_\nu^2 \dx x + \frac12\int_\Rd n_\nu^2 \dx x.
\end{align*}
Hence, by ellipticity of $A$, we deduce that $m_\nu \in L^\infty(0,T;H^1(\Rd))$.
Now, multiplying the Brinkman equation, Eq. \eqref{eq:System-Brinkman}, by $-\Delta m_\nu$, we compute
\begin{equation}
\label{eq:VelocityH1}
\begin{aligned}
    \nu\int_\Rd \div(A\nabla m_\nu) \div(\nabla m_\nu) \dx x + \int_\Rd \abs{\nabla m_\nu}^2\dx x &= -\int_\Rd n_\nu\Delta m_\nu\dx x\\
    &\leq \frac{1}{\lambda \nu} \int_\Rd n_\nu^2\dx x + \frac{\nu\lambda}{4}\int_\Rd \abs*{D^2 m_\nu}^2\dx x.
\end{aligned}
\end{equation}
For the first term on the left-hand side, we write
\begin{align*}
    \int_\Rd \div(A\nabla m_\nu) \div(\nabla m_\nu) \dx x &= \int_\Rd \nabla(A\nabla m_\nu) : D^2 m_\nu\dx x\\
    &= \int_\Rd \partial_k(A_{ij}\partial_j m_\nu) \, \partial^2_{ik}m_\nu\dx x\\
    &= \int_\Rd \partial_j(\partial_k m_\nu) A_{ij} \partial_i(\partial_k m_\nu)\dx x + \int_\Rd \partial_j m_\nu  \partial_k A_{ij} \partial^2_{ik}m_\nu\dx x\\
    &= \int_\Rd \nabla(\partial_k m_\nu) \cdot A \nabla(\partial_k m_\nu)\dx x + \int_\Rd \partial_j m_\nu  \partial_k A_{ij} \partial^2_{ik}m_\nu\dx x\\
    &\geq \lambda \int_\Rd |D^2 m_\nu|^2 \dx x + \int_\Rd \partial_j m_\nu  \partial_k A_{ij} \partial^2_{ik}m_\nu\dx x.
\end{align*}
Inserting this computation into Eq.~\eqref{eq:VelocityH1} and using the Gagliardo-Nirenberg interpolation inequality
\begin{equation*}
    \norm{\nabla m}_{L^\frac{2q}{q-2}(\Rd)} \leq \norm{\nabla m}_{L^2(\Rd)}^{\frac{q-d}{q}}\norm{D^2 m}_{L^2(\Rd)}^{\frac{d}{q}},
\end{equation*}
we find
\begin{align*}
    \frac{3\nu\lambda}{4} \int_\Rd |D^2 m_\nu|^2 \dx x &\leq C + C\int_\Rd |\nabla m_\nu| |\nabla A| |D^2 m_\nu| \dx x\\
    &\leq C + C\norm{\nabla A}_{L^q(\Rd)}\norm{\nabla m_\nu}_{L^{\frac{2q}{q-2}}(\Rd)}\norm{D^2 m_\nu}_{L^2(\Rd)}\\
    &\leq C + C\norm{\nabla A}_{L^q(\Rd)}\norm{\nabla m_\nu}_{L^2(\Rd)}^{\frac{q-d}{q}}\norm{D^2 m_\nu}_{L^2(\Rd)}^{\frac{q+d}{q}}\\
    &\leq C + C\norm{\nabla A}_{L^q(\Rd)}^{\frac{2q}{q-d}}\norm{\nabla m_\nu}_{L^2(\Rd)}^2 + \frac{\nu\lambda}{4}\norm{D^2 m_\nu}_{L^2(\Rd)}^2.
\end{align*}
Therefore, we have $D^2 m_\nu \in L^2(0,T;L^2(\Rd))$. We now write
\begin{align*}
    \norm{\nabla(A\nabla m_\nu)}_{L^2(0,T;L^2(\Rd))}^2 &\leq \norm{ A}_{L^\infty(0,T;L^\infty(\Rd))}^2\norm{D^2 m_\nu}_{L^2(0,T;L^2(\Rd))}^2\\
    &\;+ 2\norm{ A}_{L^\infty(0,T;L^\infty(\Rd))}\int_0^T\!\!\!\int_\Rd |\nabla m_\nu| |\nabla A| |D^2 m_\nu| \dx x\dx t \\
    &\; + \int_0^T\!\!\!\int_\Rd |\nabla A|^2 |\nabla m_\nu|^2 \dx x\dx t.
\end{align*}
The middle integral is finite, as shown in the calculation above. Similarly, the final integral is finite by observing
\begin{align*}
    \int_0^T\!\!\!\int_\Rd |\nabla A|^2 |\nabla m_\nu|^2 \dx x\dx t &\leq \int_0^T\norm{\nabla A}_{L^q(\Rd)}^2\norm{\nabla m_\nu}_{L^{\frac{2q}{q-2}}(\Rd)}^2\dx t\\
    &\leq\norm{\nabla m_\nu}_{L^\infty(0,T;L^2(\Rd))}^{\frac{2(q-d)}{q}}\int_0^T\norm{\nabla A}_{L^q(\Rd)}^2\norm{D^2 m_\nu}_{L^2(\Rd)}^{\frac{2d}{q}}\dx t\\
    &\leq C\int_0^T\norm{\nabla A}_{L^q(\Rd)}^{\frac{2q}{q-d}}\dx t + C\int_0^T \norm{D^2 m_\nu}_{L^2(\Rd)}^2 \dx t.
\end{align*}
We thus deduce Eq. \eqref{eq:reg-velo-1}, i.e.,
\begin{equation*}
    A\nabla m_\nu \in L^2(0,T;H^1(\Rd)).
\end{equation*}
Finally, we point out that directly from the Brinkman equation, Eq. \eqref{eq:System-Brinkman}, the regularity of the divergence of the velocity field, Eq. \eqref{eq:reg-velo-2}, follows, i.e.,
\begin{equation*}
    \div(A\nabla m_\nu) \in L^\infty(0,T;L^\infty(\Rd)).
\end{equation*}

\section{Uniform estimates}
\label{sec:Uniform}
In this section, we derive some fundamental uniform estimates, showing that most of the properties of the weak solutions stated in the previous section are stable with respect to the viscosity parameter. These will prove indispensable when passing to the limit in the entropy-dissipation equality -- the centrepiece of our analysis.
Throughout, let $(n\1_\nu, n\2_\nu, m_\nu)$ be a weak solution of System~\eqref{eq:System} and $n_\nu:= n\1_\nu + n\2_\nu$.

\begin{lemma}[$L^\infty$-control and integrability] 
\label{lem:LpBounds}
    We have $n_\nu\in L^\infty(0,T;L^1\cap L^\infty(\Rd))$ uniformly in $\nu$ as well as $0 \leq n_\nu  \leq \bar n$.
\end{lemma}
\begin{proof}
    We begin by showing the upper bound on the total density.
    Recall the equation satisfied by the total population
    \begin{equation}
    \label{eq:SumEquation2}
        \partialt {n_\nu} = \div(n_\nu A \nabla m_\nu) + n\1_\nu G\1(n_\nu) + n\2_\nu G\2(n_\nu).
    \end{equation}
    Let $(x^*, t^*)$ be maximum point of $n_\nu$ and assume that the value $n_\nu(x^*, t^*)$ is larger than $\bar n$. 
    In this point we have $\partial_t n_\nu=0$, $\nabla n_\nu=0$, and $G\i(n_\nu)<0$. Therefore, using the Brinkman equation, at $(x^*, t^*)$ there holds
    \begin{equation}
    \label{eq:Contradiction}
    \begin{aligned}
        \partialt {n_\nu} &= \nabla n_\nu \cdot A \nabla m_\nu + n_\nu \div(A \nabla m_\nu) + n\1_\nu G\1(n_\nu) + n\2_\nu G\2(n_\nu)\\
    	&< \frac{1}{\nu} (m_\nu - n_\nu)\\
    	&\leq 0.
    \end{aligned}
    \end{equation}
    Indeed, the last inequality is true, as the following argument shows. We set $C := n_\nu(x^*, t^*) $ and test the Brinkman equation by $|C - m_\nu|_-:=\max(-(C-m_\nu),0)\geq0$ to get
    $$
    	-\nu|C - m_\nu|_-  \div (A \nabla (m_\nu - C)) + |C - m_\nu|_- (m_\nu - C) = |C - m_\nu|_- (n_\nu - C) \leq 0,
    $$
    where the last inequality holds because $ n_\nu \leq C$. Thus, after integrating in space and time, we find
    $$
        \nu\int_0^T\!\!\!\int_\Rd \nabla |C - m_\nu|_- \cdot A \nabla |C - m_\nu|_- \dx x \dx t + \int_0^T\!\!\!\int_\Rd |C - m_\nu|_-^2 \dx x \dx t\leq 0.
    $$
    By ellipticity of the tensor $A$, the first term is nonnegative, and hence we conclude
    $$
    	|C - m_\nu|_-^2 = 0,
    $$
    whence
    $$
    	m_\nu \leq C.
    $$
    Therefore $m_\nu - n_\nu \leq 0$ in the maximum point. But then Eq.~\eqref{eq:Contradiction} is a contradiction, and we deduce that $n_\nu\leq \bar n$. Since $0 \leq n\1_\nu + n\2_\nu = n_\nu \leq \bar n$, we get the same $L^\infty$-control for the individual species.
    
    \medskip Now we can obtain uniform $L^1$-control by integrating Eq.~\eqref{eq:SumEquation2}, which yields
    \begin{align*}
        \ddt \int_\Rd n_\nu\dx x  \leq \max_{i=1,2}{\max_{s\in[0,\bar n]}|G\i(s)|}\int_\Rd n_\nu \dx x.
    \end{align*}
\end{proof}
Due to the previous lemma, we obtain uniform control on the velocity potential, $m_\nu$.
\begin{lemma}
    We have $m_\nu\in L^\infty(0,T;L^1\cap L^\infty(\Rd))$ uniformly in $\nu$. Moreover, $\sqrt{\nu}\nabla m_\nu \in L^\infty(0,T;L^2(\Rd))$, uniformly in $\nu$.
\end{lemma}

\begin{proof}
    The argument in the previous proof shows that $\norm{m_\nu}_{L^\infty(0,T;L^\infty(\Rd))}\leq \norm{n_\nu}_{L^\infty(0,T;L^\infty(\Rd))}$. Integrating the Brinkman equation, we obtain $\norm{m_\nu}_{L^\infty(0,T;L^1(\Rd))}=\norm{n_\nu}_{L^\infty(0,T;L^1(\Rd))}$.\\
    
    As for the last statement, we multiply the Brinkman equation by $m_\nu$ and integrate it to get
     \begin{align*}
        \lambda\nu\int_\Rd |\nabla m_\nu|^2 \dx x 
        &\leq \nu \int_\Rd \nabla m_\nu \cdot A \nabla m_\nu \dx x + \int_\Rd m_\nu^2 \dx x\\  &= \int_\Rd m_\nu n_\nu \dx x\\
        &\leq \frac12 \int_\Rd n_\nu^2 \dx x + \frac12 \int_\Rd m_\nu^2 \dx x,
     \end{align*}
     whence
     \begin{equation*}
         \sup_{t\in[0,T]}\prt*{\nu\int_\Rd |\nabla m_\nu|^2 \dx x} \leq \norm{n_\nu}_{L^\infty(0,T;L^2(\Rd))}^2.
     \end{equation*}
\end{proof}

\begin{lemma}[Dissipation control]
\label{lem:DissipationControl}
There holds
   $$
   \int_0^T\!\!\!\int_\Rd n_\nu |\nabla m_\nu|^2 \dx x \dx t \leq C,
   $$
   where $C>0$ is independent of $\nu$.
\end{lemma}
\begin{proof}
    First, using the symmetry of $A$, we observe
    \begin{align*}
        \int_\Rd n_\nu \partialt {m_\nu} \dx x
        &= \int_\Rd \prt*{-\nu\div(A\nabla m_\nu) + m_\nu}\partialt {m_\nu} \dx x\\ 
        &= \int_\Rd \fpartial t \prt*{-\nu\div(A\nabla m_\nu) + m_\nu} m_\nu \dx x - \nu\int_\Rd \nabla m_\nu \cdot \partialt A \nabla m_\nu \dx x \\
        &= \int_\Rd m_\nu \partialt {n_\nu} \dx x - \nu\int_\Rd \nabla {m_\nu} \cdot \partialt A \nabla {m_\nu} \dx x.
    \end{align*}
    Hence, using Eq.~\eqref{eq:SumEquation} for $n_\nu$ and the ellipticity assumption for $A$,
    \begin{align*}
        \frac12 \ddt \int_\Rd n_\nu m_\nu \dx x 
        &\leq -\int_\Rd n_\nu \nabla m_\nu \cdot A \nabla m_\nu \dx x + \max_i \norm{G\i}_{L^\infty([0,\bar n])}\int_\Rd n_\nu m_\nu \dx x\\
        &\quad- \frac\nu2 \int_\Rd \nabla m_\nu \cdot \partialt A \nabla m_\nu \dx x\\
        &\leq - \lambda \int_\Rd n_\nu |\nabla m_\nu|^2 \dx x + C\int_\Rd n_\nu m_\nu \dx x + \frac\nu2 \norm{\partial_t A}_{L^\infty(\Rd)}\norm{\nabla m_\nu}_{L^2(\Rd)}^2.
    \end{align*}
    At this stage, let us recall that $\sqrt{\nu}\nabla m_\nu$ is bounded in $L^2(0,T;L^2(\Rd))$. Thus, integrating in time, we find
    $$
        \int_0^T\!\!\!\int_\Rd n_\nu |\nabla m_\nu|^2 \dx x \dx t \lesssim 1 +\int_\Rd n^\mathrm{in} m^\mathrm{in} \dx x  \dx x \leq C,
    $$
    where we have used the uniform $L^\infty(0,T;L^2(\Rd))$-bounds.
\end{proof}

\begin{lemma}[Moment estimate]
\label{lem:2ndMoment}
There holds
    \begin{align*}
        \sup_{\nu > 0} \sup_{0 \leq t \leq T} \int_\Rd n_\nu |x|^2 \dx x \leq C,
    \end{align*}
    where the constant $C>0$ does not depend on $\nu$.
\end{lemma}
\begin{proof}
    This estimate follows from the previous one and the corresponding assumption on the initial data. Indeed,
    \begin{align*}
        \frac12 \ddt \int_\Rd n_\nu |x|^2 \dx x 
        &= \frac12 \int_\Rd |x|^2 \div(n_\nu A\nabla m_\nu) + \frac12 |x|^2 \prt*{n_\nu\1 G\1(n_\nu) + n_\nu\2 G\2(n_\nu)} \dx x\\
        &\leq - \int_\Rd x \cdot n_\nu (A\nabla m_\nu )\dx x + \frac12 \max_{i=1,2}\norm{G\i}_{L^\infty([0,\bar n])} \int_\Rd n_\nu |x|^2 \dx x\\
        &\leq C \int_\Rd n_\nu|x|^2 \dx x + C \int_\Rd n_\nu |A\nabla m_\nu|^2 \dx x\\
        &\leq C \int_\Rd n_\nu|x|^2 \dx x + C,
    \end{align*}
    having used the fact that $A \in L^\infty$ and the dissipation control. We conclude by Gronwall's lemma. 
\end{proof}

\begin{lemma}[Entropy estimate]
\label{lem:entropy_bounds}
There holds
    \begin{align*}
        \sup_{\nu > 0} \sup_{0 \leq t \leq T} \int_\Rd n_\nu |\log n_\nu| \dx x \leq C,
    \end{align*}
    for some constant $C>0$ independent of $\nu>0$.
\end{lemma}
\begin{proof}
    This is a simple consequence of the $L^1\cap L^\infty$-bounds and the second-moment control and uses no further property of the equation. We refer to ~\cite{DDMS2024, DMS_phenotype} for details.
\end{proof}

\begin{lemma}[Time derivative]
\label{lem:TimeDerivative}
    We have $\partial_t n_\nu \in L^2(0,T;H^{-1}(\Rd))$ uniformly in $\nu>0$. 
\end{lemma}
\begin{proof}
    Let $\xi \in L^2(0,T;H^1(\Rd))$. Then
    \begin{align*}
        \int_\Rd \xi \partialt {n_\nu}\dx x &= -\int_\Rd \prt{n_\nu A\nabla m_\nu}\cdot\nabla \xi \dx x + \sum_{i=1,2} \int_\Rd \xi n_\nu\i G\i(n_\nu)\dx x\\
        &\leq \sqrt{\bar n}\norm{A}_{L^\infty(\Rd)}\norm{\sqrt{n_\nu}\nabla m_\nu}_{L^2(\Rd)}\norm{\nabla\xi}_{L^2(\Rd)} + C\norm{\xi}_{L^2(\Rd)}.
    \end{align*}
    After integrating in time, the result follows, using Lemma~\ref{lem:DissipationControl}. 
\end{proof}

\section{The entropy equality}
\label{sec:Entropy}
As foreshadowed in the introduction, we want to exploit the dissipation of the Boltzmann-Shannon entropy functional
\begin{align*}
    \curlyH[n] = \int_\Rd n(x,t) (\log n(x,t) - 1) \, \dx x,
\end{align*}
which encodes regularity of $n_\nu$ that can be leveraged to establish its strong compactness. In consequence, this piece of information can be used to upgrade the weak convergence of the velocity field, $A\nabla m_\nu$, which, in turn, can be used to pass to the limit in System \eqref{eq:System}.

Let us begin the derivation of the entropy identity satisfied by weak solutions of System~\eqref{eq:System} by introducing a mollification of the total density equation and estimating the commutator similar to the strategy of DiPerna-Lions~\cite{DL}.
More precisely, we prove the following result:
\begin{proposition}[Entropy Equality]\label{lem:entropy_inequality}
	Let $(n\1_\nu, n\2_\nu,m_\nu)$ be any weak solution of System~\eqref{eq:System}. There holds
	\begin{equation}
		\label{eq:entropy_equality}
		\begin{split}
			\mathcal{H}[n_\nu(T)] - \mathcal{H}[n_\nu^{\mathrm{in}}]
			 & - \int_0^T \!\!\! \int_\Rd  n_\nu \div(A \nabla m_\nu) \dx x \dx t \\
			 & =\int_0^T \!\!\! \int_\Rd \log n_\nu
			 (n\1_\nu G\1(n_\nu) + n\2_\nu G\2(n_\nu)) \dx x \dx t.
		\end{split}
	\end{equation}
\end{proposition}

\begin{proof}
Let $\eta$ be a standard symmetric mollification kernel in $\R^{d+1}$ and let $\eta_\epsilon(x,t):=\epsilon^{-(d+1)}\eta(x/\epsilon, t/\epsilon)$. Then let
\begin{equation*}
    f_\epsilon(x,t):= f\star \eta_\epsilon (x,t) = \int_0^T\!\!\!\int_\Rd f(\tau,y)\eta_\epsilon(x-y,t-\tau)\dx y \dx \tau.
\end{equation*}
Let us, for simplicity, use the notation $v = -A\nabla m_\nu$ for the velocity field. Throughout this proof, we shall omit all subscript $\nu$ for the sake of clarity.
Mollifying the equation for $n$, Eq. \eqref{eq:SumEquation}, we obtain the pointwise identity
\begin{equation*}
    \partialt {n_\epsilon} = -\div(nv)_\epsilon + \prt*{n\1G\1(n)}_\epsilon + \prt*{n\2G\2(n)}_\epsilon.
\end{equation*}
Multiplying by $\log(n_\epsilon+\alpha)$, for some $\alpha>0$, and integrating, we obtain
\begin{equation}
\label{eq:MollifiedEquation}
\begin{aligned}
    &\int_\Rd \prt{n_{\epsilon}(T)+\alpha}\prt{\log(n_{\epsilon}(T)+\alpha) - 1} \dx x - \int_\Rd \prt{n_{\epsilon}(0)+\alpha}\prt{\log(n_{\epsilon}(0)+\alpha) - 1} \dx x\\
    &= -\int_0^T\!\!\!\int_\Rd \log(n_\epsilon+\alpha)\div(n v)_\epsilon\dx x\dx t + \sum_{i=1,2}\int_0^T\!\!\!\int_\Rd \log(n_\epsilon+\alpha)\prt*{n\i G\i(n)}_\epsilon\dx x\dx t.
\end{aligned}
\end{equation}
Since $n\in C([0,T];L^2(\Rd))$, we have $n_\epsilon(t)\to n(t)$ in $L^2(\Rd)$ for every $t\in[0,T]$. Hence, the left-hand side of Eq.~\eqref{eq:MollifiedEquation} converges to
\begin{equation*}
    \int_\Rd \prt{n(T)+\alpha}\prt{\log(n(T)+\alpha) - 1} \dx x - \int_\Rd \prt{n^{\mathrm{in}}+\alpha}\prt{\log(n^{\mathrm{in}}+\alpha) - 1} \dx x,
\end{equation*}
as $\epsilon \to 0$. For the transport term, we have
\begin{align*}
    \int_0^T\!\!\!\int_\Rd \log(n_\epsilon+\alpha)\div(n v)_\epsilon\dx x\dx t &= \int_0^T\!\!\!\int_\Rd \log(n_\epsilon+\alpha)\div(n_\epsilon v)\dx x\dx t\\
    &\quad+ \int_0^T\!\!\!\int_\Rd \log(n_\epsilon+\alpha)\brk*{\div(n v)_\epsilon - \div(n_\epsilon v)}\dx x\dx t.
\end{align*}
Let us first discuss the commutator term. Notice that $\log(n_\epsilon+\alpha) \in L^\infty(0,T;L^\infty(\Rd))$ uniformly in $\epsilon$. Since $v\in L^2(0,T;H^1(\Rd))$ and $n\in L^\infty(0,T;L^2(\Rd))$, we can apply~\cite[Lemma~2.3]{Lions_Incompressible}, to deduce
\begin{align*}
    \norm{\div(n v)_\epsilon - \div(n_\epsilon v)}_{L^1(0,T;L^1(\Rd))}\to 0,
\end{align*}
as $\epsilon \to 0$. Thus, the commutator term vanishes in the limit. For the remaining term, we write
\begin{align*}
   \int_0^T\!\!\!\int_\Rd \log(n_\epsilon+\alpha)\div(n_\epsilon v)\dx x\dx t &= -\int_0^T\!\!\!\int_\Rd \frac{n_\epsilon}{n_\epsilon+\alpha}\nabla n_\epsilon\cdot v \dx x\dx t \\
   &= -\int_0^T\!\!\!\int_\Rd \nabla n_\epsilon\cdot v \dx x\dx t + \alpha\int_0^T\!\!\!\int_\Rd \frac{1}{n_\epsilon+\alpha}\nabla n_\epsilon\cdot v \dx x\dx t\\
   &= \int_0^T\!\!\!\int_\Rd n_\epsilon\div v \dx x\dx t - \alpha\int_0^T\!\!\!\int_\Rd \log{(n_\epsilon+\alpha)}\div v \dx x\dx t\\
   &\to \int_0^T\!\!\!\int_\Rd n\div v \dx x\dx t - \alpha\int_0^T\!\!\!\int_\Rd \log{(n+\alpha)}\div v \dx x\dx t.
\end{align*}
Passing to the limit $\epsilon\to 0$ in Eq.~\eqref{eq:MollifiedEquation}, we therefore obtain
\begin{equation*}
\label{eq:MollifiedEquation2}
\begin{aligned}
    &\int_\Rd \prt{n(T)+\alpha}\prt{\log(n(T)+\alpha) - 1} \dx x - \int_\Rd \prt{n^{\mathrm{in}}+\alpha}\prt{\log(n^{\mathrm{in}}+\alpha) - 1} \dx x\\
    &= -\int_0^T\!\!\!\int_\Rd n\div v \dx x\dx t + \sum_{i=1,2}\int_0^T\!\!\!\int_\Rd \log(n+\alpha)n\i G\i(n)\dx x\dx t\\
    &\quad +\alpha\int_0^T\!\!\!\int_\Rd \log{(n+\alpha)}\div v \dx x\dx t.
\end{aligned}
\end{equation*}
Since $\div v = \div(A\nabla m) \in L^1(0,T;L^1(\Rd))$, the last term vanishes in the limit $\alpha\to 0$. In other words, we can use the monotone convergence theorem 
\begin{align*}
    &\abs*{\int_0^T\!\!\!\int_\Rd \log(n+\alpha)n\i G\i(n)\dx x\dx t - \int_0^T\!\!\!\int_\Rd \log n\;n\i G\i(n)\dx x\dx t} \\
    &\leq \int_0^T\!\!\!\int_\Rd \abs*{\log(n+\alpha)-\log{n}} n\i \abs{G\i(n)}\dx x\dx t\\
    &\leq C\int_0^T\!\!\!\int_\Rd \abs*{\log(n+\alpha)-\log{n}}\dx x\dx t.
\end{align*}
The sequence of nonnegative functions $f_\alpha(x,t) := \abs*{\log(n(x,t)+\alpha)-\log{n(x,t)}}$ is decreasing as $\alpha\to 0$ and converges pointwise to $0$ in $\Rd\times(0,T)$. The same is done with the final and initial times.
This finally leads to the identity:
\begin{align*}
        &\int_\Rd n(T)\prt{\log{n(T)}-1} \dx x - \int_\Rd n^{\mathrm{in}}\prt{\log{n^{\mathrm{in}}} - 1} \dx x - \int_0^T\!\!\!\int_\Rd n\div\prt{A\nabla m} \dx x\dx t\\
    &= \int_0^T\!\!\!\int_\Rd \log{n}\;n\1 G\1(n)\dx x\dx t + \int_0^T\!\!\!\int_\Rd \log{n}\; n\2 G\2(n)\dx x\dx t,
\end{align*}
as required.
\end{proof}

An immediate consequence of the entropy identity is the following uniform bound.
\begin{corollary}
\label{cor:Dissipation}
    There holds
    \begin{equation*}
        -\int_0^T\!\!\!\int_\Rd n_\nu\div(A\nabla m_\nu) \dx x \dx t \leq C,
    \end{equation*}
    where the constant $C$ is independent of $\nu$.
\end{corollary}
\begin{proof}
    From Eq. ~\eqref{eq:entropy_equality} we have
    \begin{align*}
        -\int_0^T\!\!\!\int_\Rd n_\nu\div(A\nabla m_\nu) \dx x \dx t &\leq C\int_0^T\!\!\!\int_\Rd n_\nu\abs{\log{n_\nu}} \dx x \dx t + \int_\Rd n^{\mathrm{in}}\abs{\log{n^{\mathrm{in}}}} \dx x\\
        &\quad+ \int_\Rd n_\nu(T)\abs{\log{n_\nu(T)}} \dx x + \norm{n^{\mathrm{in}}}_{L^1(\Rd)} + \norm{n(T)}_{L^1(\Rd)},
    \end{align*}
    which is finite, independently of $\nu$, due to Lemma~\ref{lem:entropy_bounds}
\end{proof}

\section{Strong convergence of the total density}
\label{sec:StrongConvergence}
Having established uniform estimates and the entropy dissipation equality, we now focus on proving the total population's strong compactness $n_\nu$. To achieve this, we first employ compactness arguments exploiting the uniform bounds from the entropy dissipation derived earlier. Henceforth, for any $\nu>0$, we consider a weak solution $(n_\nu\1, n_\nu\2,m_\nu)$ of System~\eqref{eq:System} and we wish to perform the limit $\nu\to0$ along a suitable subsequence. Before addressing the strong compactness, let us note that, by the uniform bounds from above, we have, up to a subsequence,
$$
	n\i_\nu \rightharpoonup n\i_0,\quad\quad n_\nu \rightharpoonup n_0:=n_0\1 + n_0\2,
$$
weakly in $L^2(0,T;L^2(\Rd))$ for some $n\i_0 \in L^\infty(0,T; L^1\cap L^\infty(\Rd))$, $i=1,2$. Furthermore, $m_\nu$ has the same weak limit. Indeed, 
for any $\varphi \in  C_c^\infty(\Rd\times[0,T])$, we find
	\begin{align*}
		&\int_0^T\!\!\!\int_\Rd \prt{-\nu \div (A \nabla m_\nu)} \varphi\dx x \dx t\\
		&= \nu\int_0^T\!\!\!\int_\Rd A \nabla m_\nu  \cdot \nabla \varphi  \dx x \dx t\\
		&= \nu \int_0^T\!\!\!\int_\Rd \div (A \nabla \varphi) m_\nu \dx x \dx t\\
        &\leq \nu \norm{m_\nu}_{L^\infty(0,T; L^\infty(\Rd))} \prt*{\norm{A}_{L^\infty(0,T; L^\infty(\Rd))} + \norm{\nabla A}_{L^{\frac{2q}{q-d}}(0,T;L^q(\Rd))}}C(\varphi)\\
        &\leq C\nu,
	\end{align*}
which vanishes as $\nu\to 0$, for any $\varphi$. It follows from
\begin{equation*}
    -\nu \div (A \nabla m_\nu) + m_\nu = n_\nu
\end{equation*}
that $m_\nu\rightharpoonup n_0$ weakly in $L^2(0,T;L^2(\Rd))$. Having identified the weak limit of $m_\nu$, we will now explain how uniform bounds implied by the entropy equality allow us to infer strong convergence.

\begin{lemma}[$m_\nu$ is strongly compact]
\label{lem:M_Strong}
	Up to a subsequence, we have
	$$
		m_\nu \to n_0,\quad \nabla m_\nu\rightharpoonup \nabla n_0
	$$
	as $\nu \to 0$, in $L^2(0,T;L^2(\Rd))$.
\end{lemma}
\begin{proof}
	The argument is based on the Aubin-Lions lemma.\\ 

    \textbf{Space regularity.} 
    To this end, consider
	\begin{align*}
		\lambda \int_\Rd | \nabla m_\nu |^2 \dx x \dx t 
		&\leq \int_\Rd \nabla m_\nu \cdot A \nabla m_\nu \dx x \dx t\\
		&= - \int_\Rd m_\nu \, \div (A \nabla m_\nu)\dx x \dx t\\
		&= - \int_\Rd n_\nu \, \div (A \nabla m_\nu) \dx x \dx t+ \int_\Rd (n_\nu - m_\nu) \, \div (A \nabla m_\nu)\dx x \dx t\\
		&\leq - \int_\Rd n_\nu \, \div (A \nabla m_\nu)\dx x\dx t\\
		&\leq C,
	\end{align*}
	by the entropy bound, Corollary~\ref{cor:Dissipation}. From this uniform gradient control, we deduce that $n_0\in L^2(0,T;H^1(\Rd))$ and, after extracting a subsequence, $\nabla m_\nu \rightharpoonup \nabla n_0$.\\ 

    \textbf{Time regularity.} We begin by introducing suitable test functions.
	Let $\varphi \in H^1(\Rd)$ be arbitrary and let $\xi$ solve
	$$
		-\nu \div (A \nabla \xi) + \xi = \varphi.
	$$
	We will show that the solution is in $H^1(\Rd)$ uniformly in $\nu$. Indeed, testing this equation by $\xi$ immediately gives
	$$
		\nu \int_\Rd \nabla \xi \cdot A \nabla \xi \dx x + \int_\Rd |\xi|^2 \dx x = \int_\Rd \varphi \xi \dx x, 
	$$
	whence $\norm{\xi}_{L^2(\Rd)} \leq \norm{\varphi}_{L^2(\Rd)}$. In particular, this control does not depend on $\nu$. Next, let us test the equation by $-\div(A \nabla \xi)$ to get
	$$
		\nu \int_\Rd | \div (A \nabla \xi) |^2 \dx x+ \int_\Rd \nabla \xi \cdot A \nabla \xi \dx x = \int_\Rd \nabla \varphi \cdot A \nabla \xi \dx x.
	$$
	From ellipticity of $A$ we get $\norm{\nabla \xi}_{L^2(\Rd)} \leq \lambda^{-1} \norm{A(t)}_{L^\infty(\Rd)} \norm{\nabla \varphi}_{L^2(\Rd)}$, and it is independent of $\nu>0$.\\
	
	Now, let us estimate the $H^{-1}$-norm of $m_\nu$:	
	\begin{align*}
		\int_\Rd \varphi \partialt {m_\nu} \dx x 
		&= \int_\Rd (- \nu \div (A \nabla \xi) + \xi) \partialt {m_\nu}\dx x\\
		&= \int_\Rd \fpartial t (-\nu \div (A \nabla m_\nu) + m_\nu) \xi \dx x - \nu\int_\Rd \nabla\xi \cdot\prt*{\partialt A \nabla m_\nu} \dx x\\
		&= \int_\Rd \partialt {n_\nu}\; \xi\dx x - \nu\int_\Rd \nabla\xi \cdot\prt*{\partialt A \nabla m_\nu} \dx x\\
        &\leq \norm*{\partialt {n_\nu}}_{H^{-1}(\Rd)} \norm{\xi}_{H^1(\Rd)} + \norm*{\partialt A(t)}_{L^\infty(\Rd)}\norm{\nabla m_\nu (t)}_{L^2(\Rd)}\norm{\nabla \xi}_{L^2(\Rd)}\\
		&\leq C(t) \prt*{\norm*{\partialt {n_\nu}}_{H^{-1}(\Rd)}+1} \norm{\varphi}_{H^1(\Rd)},
	\end{align*}
    where $C(t)\in L^\infty(0,T)$.
    Dividing by $\norm{\varphi}_{H^1(\Rd)}$ and passing to the supremum we find 
    $$
        \norm*{\partialt {m_\nu}}_{H^{-1}(\Rd)} \leq C + C \norm*{\partialt {n_\nu}}_{H^{-1}(\Rd)}.
    $$
    Squaring and integrating in time gives
	$$
		\norm*{\partialt {m_\nu}}_{L^2(0,T;H^{-1}(\Rd))} \leq C + C \norm*{\partialt {n_\nu}}_{L^2(0,T;H^{-1}(\Rd))} \leq C,
	$$
    using Lemma~\ref{lem:TimeDerivative}, where the constant $C$ does not depend on $\nu$. \\
	
    With this space and time regularity, we can apply the Aubin-Lions lemma to deduce that there exists a subsequence such that
	$$
		m_\nu \rightarrow n_0\quad \text{in } L^2(0,T; L^2_{\mathrm{loc}}(\Rd)).
	$$
 We extend to global convergence by using the uniform second-moment bound in Lemma~\ref{lem:2ndMoment}.
\end{proof}

\begin{lemma}[Strong convergence of $n_\nu$]
    
\end{lemma}
\begin{proof}
This follows from Lemma~\ref{lem:M_Strong} and the boundedness of the entropy dissipation term. Indeed
    \begin{align*}
        \int_0^T\!\!\!\int_\Rd |m_\nu-n_\nu|^2 \dx x \dx t 
        &= \nu\int_0^T\!\!\!\int_\Rd (m_\nu-n_\nu) \div(A\grad m_\nu) \dx x \dx t\\
        &=- \nu\int_0^T\!\!\!\int_\Rd n_\nu  \div(A \nabla m_\nu) + \nabla m_\nu \cdot A \nabla m_\nu \dx x\dx t\\
        &\leq C \nu.
    \end{align*}
The statement follows from the triangulation
    $$
        \norm{n_\nu - n_0}_{L^2(0,T;L^2(\Rd))}^2 \lesssim  \nu + \norm{m_\nu - n_\nu}_{L^2(0,T;L^2(\Rd))}^2 \to 0,
    $$
    as $\nu \to 0$.
\end{proof}
Extracting a subsequence, we can guarantee that $n_\nu\to n_0$ almost everywhere in $\Rd\times(0,T)$.

\section{Limit equation for the total population}
\label{sec:Limit}
With the regularity from the sections before, we are now able to pass to the limit in the weak form of Eq.~\eqref{eq:SumEquation},
\begin{align*}
	-\int_0^T\!\!\!\int_\Rd  n_\nu \partialt\varphi \dx x \dx t &- \int_\Rd \phi(0) n^{\mathrm{in}}\dx x + \int_0^T\!\!\!\int_\Rd \nabla \varphi \cdot A \nabla m_\nu \dx x \dx t\\
    &= \int_0^T\!\!\!\int_\Rd (n_\nu\1 G\1 (n_\nu) + n_\nu\2 G\2 (n_\nu)) \varphi \dx x\dx t,
\end{align*}
to obtain
\begin{equation}
\label{eq:WeakDarcy}
\begin{aligned}
	-\int_0^T\!\!\!\int_\Rd n_0 \partialt\varphi \dx x \dx t &- \int_\Rd \phi(0) n^{\mathrm{in}}\dx x  + \int_0^T\!\!\!\int_\Rd \nabla \varphi \cdot A \nabla n_0 \dx x \dx t \\
    &= \int_0^T\!\!\!\int_\Rd (n_0\1 G\1 (n_0) + n_0\2 G\2 (n_0)) \varphi \dx x\dx t,
\end{aligned}
\end{equation}
for every $\varphi \in C_c^\infty(\Rd\times[0,T))$, i.e., the limit $n_0$ is a weak solution of
\begin{equation*}
    \partialt {n_0} - \div(n_0A\nabla n_0) = n_0\1 G\1 (n_0) + n_0\2 G\2 (n_0),
\end{equation*}
with initial data $n^\mathrm{in}:=n^{(1),\mathrm{in}}+n^{(2),\mathrm{in}}$. Let us point out that, at this stage, $n\i_0$, for $i=1,2$, are merely the weak limits, and it is not clear that they satisfy System \eqref{eq:InviscidSystem}. Next, we recall the following regularity properties of the limit total density $n_0$:
\begin{align*}
    0\leq n_0 \leq \bar n, \quad n_0 \in L^\infty(0,T;L^1\cap L^\infty(\Rd))\cap L^2(0,T;H^1(\Rd)).
\end{align*}
Moreover, from Lemma~\ref{lem:TimeDerivative}, we deduce
\begin{equation*}
    \partialt {n_0} \in L^2(0,T;H^{-1}(\Rd)),
\end{equation*}
and, applying Fatou's lemma to the results of Lemma~\ref{lem:2ndMoment} and Lemma~\ref{lem:entropy_bounds},
\begin{equation*}
    \sup_{t\in[0,T]}\int_\Rd |x|^2 n_0\dx x \leq C,\quad \sup_{t\in[0,T]}\int_\Rd n_0\abs{\log{n_0}}\dx x \leq C.
\end{equation*}
As a consequence of the $H^1$ and $H^{-1}$-bounds, we deduce that $n_0\in C([0,T];L^2(\Rd))$, and that
\begin{equation*}
    -\int_0^T\!\!\!\int_\Rd n_0 \partialt\varphi \dx x \dx t - \int_\Rd \phi(0) n^{\mathrm{in}}\dx x = \int_0^T \langle \partialt {n_0}, \varphi\rangle_{H^{-1},H^1} \dx t.
\end{equation*}
Using this identity, the $L^\infty$-bounds, and the $H^1$-bound, we can extend the class of test function in~\eqref{eq:WeakDarcy} to $\varphi\in L^2(0,T;H^1(\Rd))$.\\

Now, we use $\varphi = \log n_0$ as a test function\footnote{In fact, to guarantee $\varphi \in L^2(0,T;H^1(\Rd))$ we need to use $\varphi(x,t)=\log(n_0(x,t)+\alpha)\phi_R(x)$ for $\alpha> 0$ and a smooth cut-off function $\phi_R$. We leave the details of the limit procedures $\alpha\to0$ and $R\to\infty$ to the reader or refer to an analogue argument carried out in the proof of Proposition 2.6 in \cite{DDMS2024}.} and get the entropy identity
\begin{equation}
\label{eq:DarcyEntropy}
    \begin{aligned}
	\mathcal{H}[n_0(T)] &- \mathcal{H}[n^\mathrm{in}] + \int_0^T \int_\Rd  \nabla n_0 \cdot A \nabla n_0 \dx x \dx t\\
    &= \int_0^T\!\!\!\int_\Rd \log n_0 \prt*{n_0\1G\1(n_0) + n_0\2 G\2(n_0)} \dx x \dx t.
    \end{aligned}
\end{equation}

\section{Strong convergence of $\nabla m$}
\label{sec:StrongVelocity}

In this section, we use the two entropy identities to upgrade the weak convergence $\nabla m_\nu \rightharpoonup \nabla n_0$ to convergence in norm. This will allow us to pass to the limit in the equations for the individual densities $n\i_\nu$, System \eqref{eq:InviscidSystem}. First, we observe that
\begin{align*}
    \int_0^T\!\!\!\int_\Rd  \nabla m_\nu \cdot A \nabla m_\nu \dx x \dx t &= -\int_0^T\!\!\!\int_\Rd   n_\nu \, \div (A \nabla m_\nu) \dx x \dx t\\
        &\quad+ \int_0^T\!\!\!\int_\Rd (n_\nu-m_\nu) \, \div (A \nabla m_\nu) \dx x \dx t\\
        &\leq -\int_0^T\!\!\!\int_\Rd   n_\nu \, \div (A \nabla m_\nu) \dx x \dx t,
\end{align*}
as the second term is nonpositive due to Brinkman's equation.
Let us denote for brevity
\begin{equation*}
    \text{React}_\nu = \int_0^T\!\!\!\int_\R \log n_\nu\prt*{n_\nu\1 G\1(n_\nu)+n_\nu\2 G\2(n_\nu)} \dx x \dx t,
\end{equation*}
and, correspondingly, $\text{React}_0$.
Then, using the entropy identities, Eq.~\eqref{eq:entropy_equality} and Eq.~\eqref{eq:DarcyEntropy},
\begin{equation}
\label{eq:EntropyIntermediate}
\begin{aligned}
	\int_0^T\!\!\!\int_\Rd  \nabla m_\nu \cdot A \nabla m_\nu \dx x \dx t 
	&\leq -\int_0^T\!\!\!\int_\Rd   n_\nu \, \div (A \nabla m_\nu) \dx x \dx t\\
	&= \curlyH[n^\mathrm{in}] - \curlyH[n_\nu(T)] + \text{React}_\nu\\
	&= \curlyH[n_0(T)] - \curlyH[n_\nu(T)] + \text{React}_\nu -  \text{React}_0\\
    &\qquad+ \int_0^T\!\!\!\int_\Rd \nabla n_0 \cdot A \nabla n_0 \dx x \dx t.
\end{aligned}
\end{equation}
We now wish to pass to the limit in this inequality. Using the uniform $L^p$-bounds and the weak forms of Eq.~\eqref{eq:SumEquation} and Eq.~\eqref{eq:WeakDarcy}, it can be easily shown that $n_\nu(T)\rightharpoonup n_0(T)$ weakly in $L^1(\Rd)$. We refer to~\cite[Proposition~4.4]{DMS_phenotype} for more details. It then follows that
\begin{equation*}
    \limsup_{\nu\to 0}\prt*{\curlyH[n_0(T)] - \curlyH[n_\nu(T)]} \leq 0,
\end{equation*}
by convexity of the entropy functional.
Likewise, using the same procedure as in~\cite[Section~4.2]{DDMS2024} and~\cite[Section~4]{DMS_phenotype}, we can show
\begin{equation*}
    \text{React}_\nu \to \text{React}_0.
\end{equation*}
Thus, Eq. \eqref{eq:EntropyIntermediate} yields
$$
	\limsup_{\nu \to 0} \int_0^T\!\!\!\int_\Rd \nabla m_\nu \cdot A \nabla m_\nu \dx x \dx t\leq \int_0^T\!\!\!\int_\Rd\nabla n_0 \cdot A \nabla n_0 \dx x \dx t.
$$
At the same time, since $A$ defines an inner product on $L^2(\Rd\times(0,T))$ ($A$ being symmetric and positive definite) and since $\nabla m_\nu\rightharpoonup\nabla n_0$, 
$$
	\liminf_{\nu \to 0}  \int_0^T\!\!\!\int_\Rd \nabla m_\nu \cdot A \nabla m_\nu \dx x \dx t \geq \int_0^T\!\!\!\int_\Rd \nabla n_0 \cdot A \nabla n_0 \dx x \dx t.
$$
Thus, in conjunction with the ellipticity of $A$, from
$$
	\lambda \norm{\nabla m_\nu - \nabla n_0}_{L^2(0,T;L^2(\Rd))}^2 \leq  \int_0^T \int_\Rd (\nabla m_\nu - \nabla n_0) \cdot A (\nabla m_\nu - \nabla n_0) \dx x\dx t\to 0,
$$
we may infer
$$
	\nabla m_\nu \to \nabla n_0,
$$
strongly in $L^2(0,T;L^2(\Rd))$.
With this information, the proof of Theorem~\ref{thm:Main} is concluded by passing to the limit in the weak formulation of the equations for $n_\nu\i$ in Definition~\ref{def:WeakBrinkman}.

\section*{Acknowledgements}
T.D. acknowledges the support of the National Science Centre, Poland, project no.\\ 2023/51/D/ST1/02316.
M.S. would like to acknowledge the support of the German Academic Exchange Service (DAAD) Project ID 57699543.

\bibliographystyle{abbrv}
\bibliography{lit}

\end{document}